\documentclass[11pt,a4paper]{amsart}

\usepackage{amsmath}
\usepackage{mathtools}
\usepackage{amsthm}
\usepackage{amssymb}
\usepackage{hyperref}

\newcommand*\fullref[3][\relax]{%
  \ifdefined\hyperref%
    {\hyperref[#3]{#2\penalty 200\ \ref*{#3}#1}}%
  \else%
    {#2\penalty 200\ \relax\ref{#3}#1}%
  \fi%
}

\newcommand*{\defterm}[1]{\emph{#1}}

\makeatletter
\newcommand\chyph{\penalty\@M-\hskip\z@skip}
\makeatother

\theoremstyle{definition}
\newtheorem{definition}{Definition}[section]

\newtheorem{example}[definition]{Example}

\newtheorem{question}[definition]{Question}

\theoremstyle{plain}

\newtheorem{lemma}[definition]{Lemma}
\newtheorem{proposition}[definition]{Proposition}
\newtheorem{theorem}[definition]{Theorem}

\numberwithin{equation}{section}
\DeclarePairedDelimiter{\parens}{\lparen}{\rparen}
\DeclarePairedDelimiter{\bracks}{\lbrack}{\rbrack}

\DeclarePairedDelimiter{\set}{\{}{\}}
\DeclarePairedDelimiterX{\gset}[2]{\{}{\}}{\,#1:#2\,}

\newcommand*{\biggg}{\bBigg@{4}}

\newcommand*{\Biggg}{\bBigg@{5}}

\makeatletter
\newcommand*{\sizeddelimiter}[2]{\bBigg@{#1}#2}
\makeatother

\makeatletter
\newcommand*{\sizedsurd}[2][]{%
  {\@mathmeasure\z@{\nulldelimiterspace\z@}%
     {\sqrt[#1]{\vcenter to #2\big@size{}}}%
     \box\z@}%
 }
\makeatother

\DeclareMathOperator{\im}{im}

\newcommand*{\nset}{\mathbb{N}}
\newcommand*{\zset}{\mathbb{Z}}

\newcommand*{\emptyword}{\varepsilon}
\newcommand*{\rev}{\mathrm{rev}}

\DeclarePairedDelimiterX{\pres}[2]{\langle}{\rangle}{#1\,\delimsize\vert\,\mathopen{}#2}

\newcommand*{\drel}[1]{\mathcal{#1}}

\newcommand*{\N}{\mathcal{N}}

\newcommand*\A{\mathcal{A}}
\newcommand*\B{\mathcal{B}}
\newcommand*\C{\mathcal{C}}

\renewcommand*\P{\mathcal{P}}

\newcommand\CF{\mathcal{CF}}
\newcommand\DCF{\mathcal{DCF}}
\newcommand\cfwp{U(\CF)}
\newcommand\dcfwp{U(\DCF)}
\newcommand\WP{\operatorname{WP}}

\newcommand\BR{\mathrm{BR}}

\newcommand\ra{\rightarrow}

\newcommand*\pda[1]{\mathcal{#1}}

\begin{document}

\title{Context-free word problem semigroups}

\author{Tara Brough}
\address{
Centro de Matem\'{a}tica e Aplica\c{c}\~{o}es\\
Faculdade de Ci\^{e}ncias e Tecnologia\\
Universidade Nova de Lisboa\\
2829--516 Caparica\\
Portugal
}
\email{tarabrough@gmail.com}

\thanks{The first author was supported by the {\scshape FCT} (Funda\c{c}\~{a}o para a Ci\^{e}ncia e a
  Tecnologia/Portuguese Foundation for Science and Technology) fellowship {\scshape SFRH}/\allowbreak {\scshape
    BPD}/121469/2016 and by the FCT project {\scshape UID}/Multi/04621/2013.}

\author{Alan J. Cain}
\address{%
Centro de Matem\'{a}tica e Aplica\c{c}\~{o}es\\
Faculdade de Ci\^{e}ncias e Tecnologia\\
Universidade Nova de Lisboa\\
2829--516 Caparica\\
Portugal
}
\email{%
a.cain@fct.unl.pt
}

\thanks{The second author was supported by an Investigador {\scshape FCT} fellowship ({\scshape IF}/01622/2013/{\scshape
    CP}1161/{\scshape CT}0001).}

\thanks{For the first and second authors, this work was partially supported by FCT projects {\sc UID}/{\sc
    MAT}/00297/2019 (Centro de Matem\'{a}tica e Aplica\c{c}\~{o}es), {\scshape PTDC}/{\scshape MHC-FIL}/2583/2014 and
  {\scshape PTDC}/{\scshape MAT-PUR}/31174/2017.}

\author{Markus Pfeiffer}
\address{%
School of Computer Science\\
University of St Andrews \\
North Haugh, St Andrews\\
Fife KY16 9SX\\
United Kingdom
}
\email{markus.pfeiffer@st-andrews.ac.uk}

\thanks{This work was started during a visit by the third author to the Universidade Nova de Lisboa, which was supported
  by the exploratory project {\scshape IF}/01622/2013/\allowbreak{\scshape CP}1161/{\scshape CT}0001 attached to the
  second author's research fellowship.}

\begin{abstract}
  This paper studies the classes of semigoups and monoids with context-free and deterministic context-free word
  problem. First, some examples are exhibited to clarify the relationship between these classes and their connection
  with the notions of word-hyperbolicity and automaticity. Second, a study is made of whether these classes are closed
  under applying certain semigroup constructions, including direct products and free products, or under regressing from
  the results of such constructions to the original semigroup(s) or monoid(s).
\end{abstract}

\maketitle

\section{Introduction}

The deep connections between formal language theory and group theory are perhaps most clearly evidenced by the famous
1985 theorem of Muller and Schupp, which says that a group has context-free word problem if and only if it is virtually
free \cite{muller_contextfree,dunwoody_accessibility}; indeed, virtually free groups have \emph{deterministic}
context-free word problem. Since then, many studies have analyzed the classes of groups with word problems in various
families of formal languages. Herbst \& Thomas characterized the groups with one-counter word problem
\cite[Theorem~5.1]{herbst_onecounter}. (For a later elementary proof of this result, see \cite{holt_onecounter}.) The
first author of the present paper investigated groups whose word problem is an intersection of finitely many
context-free languages \cite{brough_groups,brough_phd}. Holt et al.\ studied the class of groups whose co-word problem is
context-free \cite{holt_cfcoword} and Holt and R\"{o}ver studied the the class of groups whose co-word problem is
indexed \cite{holt_indexedcoword}.

The word problem of a group is the language of words representing the identity over some set of generators and their
inverses. Thus two words $u$ and $v$ are equal in a group $G$ if and only if $uV$ is in the word problem, where $V$ is
obtained from $v$ by replacing each symbol by its inverse and reversing the word. A natural question is how to
generalize this definition to semigroups. Duncan and Gilman \cite[Definition~5.1]{duncan_hyperbolic} defined the word
problem of a semigroup $S$ with respect to a generating set $A$ to be
\begin{equation}
  \label{eq:sgwordprobdef}
  \WP(S,A) = \gset[\big]{u\#v^\rev}{u,v \in A^+, u =_S v},
\end{equation}
where $v^\rev$ is the reverse of $v$.
This definition fits well with the group definition and is natural when considering word problems recognizable
by automata equipped with a stack. It was used by Holt, Owens, and Thomas in their study of groups and semigroups with
one-counter word problem \cite{holt_onecounter}, and by Hoffmann et al.\ in their study of semigroups with context-free
word problem \cite{hoffmann_contextfree}.

The main conclusions of Hoffmann et al.'s earlier study were the result that the class of semigroups with context-free
word problem is closed under passing to finite Rees index subsemigroups and extensions
\cite[Theorem~1]{hoffmann_contextfree} and a characterization of completely simple semigroups with context-free word
problem as Rees matrix semigroups over virtually free groups \cite[Theorem~2]{hoffmann_contextfree}.

This paper explores new directions in the study of the class of semigroups with context-free word problem, including
monoids with context-free word problem, and also considers the classes of semigroups and monoids with deterministic
context-free word problem. First, \fullref{Section}{sec:examples} exhibits some natural classes of semigroups and
monoids that lie within and outside these classes; in particular \fullref{Example}{eg:notdetcfwp} shows that having
context-free and deterministic context-free word problem do not coincide for semigroups or monoids, unlike (as noted
above) for groups. \fullref{Section}{sec:wordhyp} discusses connections with the theories of word-hyperbolic and
automatic semigroups: any semigroup or monoid with context-free word problem is word-hyperbolic, but there are
non-automatic semigroups that have context-free word problem. The remainder of the core of the paper
(\fullref{Sections}{sec:dirprod}--\ref{sec:reesmatrix}) focusses on various constructions: direct products, free
products, strong semilattices of semigroups, Rees matrix semigroups and Bruck--Reilly extensions. For each
construction, the questions of interest are: (1) Are the classes of semigroups and monoids with context-free or
deterministic context-free word problem closed under that construction? (2) If the result of applying such a
construction lies in one of these classes, must the original semigroup(s) or monoids(s) lie in that same class? Finally,
\fullref{Section}{sec:openprob} lists some open problems.

\section{Preliminaries}

The \defterm{word problem} for a semigroup $S$ is defined as \eqref{eq:sgwordprobdef} above. Similarly, the
\defterm{word problem} for a monoid $M$ with respect to a generating set $A$ is the language
\begin{equation}
  \label{eq:monwordprobdef}
  \WP(M,A) = \gset[\big]{u\#v^\rev}{u,v \in A^*, u =_M v}.
\end{equation}

\begin{proposition}[{\cite[Proposition~8]{hoffmann_contextfree}}]
  \label{prop:changegensubsemigroups}
  Let $\mathfrak{C}$ be a class of languages closed under inverse homomorphisms and intersection with regular
  languages. Then
  \begin{enumerate}
  \item If a semigroup or monoid has word problem in $\mathfrak{C}$ with respect to some generating set, then it has
    word problem in $\mathfrak{C}$ with respect to any generating set.
  \item The class of semigroups (resp.\ monoids) with word problem in $\mathfrak{C}$ is closed under taking finitely
    generated subsemigroups (resp.\ submonoids).
  \end{enumerate}
\end{proposition}

The preceding result applies in particular when $\mathfrak{C}$ is the class of context-free or deterministic
context-free languages \cite{ginsburg_deterministic,hopcroft_automata}.

If a semigroup (resp.\ monoid) has word problem in a class of languages $\mathfrak{C}$,
it is said to be a $U(\mathfrak{C})$ semigroup (resp.\ monoid). We denote the classes of context-free
and deterministic context-free languages by $\CF$ and $\DCF$ respectively.
The `$U$' notation is because
\eqref{eq:sgwordprobdef} and \eqref{eq:monwordprobdef} treat the word problem as an `unfolded' relation rather than a
`two-tape' relation; see \cite{bcm_relationlanguages} for a systematic study.

\section{Examples}
\label{sec:examples}

We recall some less commonly-used terms from the theory of rewriting systems; see \cite{book_srs} for general background. A
rewriting system $(A,\drel{R})$ is \defterm{monadic} if it is length-reducing and the right-hand side of each rewrite
rule in $\drel{R}$ lies in $A \cup \{\emptyword\}$. A monadic rewriting system $(A,\drel{R})$ is \defterm{regular}
(respectively, \defterm{context-free}) if, for each $a \in A \cup\{\emptyword\}$, the set of all left-hand sides of
rewrite rules in $\drel{R}$ with right-hand side $a$ is a regular (respectively, context-free) language.

\begin{theorem}[{\cite[Theorem~3.1]{cm_wordhypunique}}]
  \label{thm:contextfreesrs}
  Let $(A,\drel{R})$ be a confluent context-free monadic rewriting system. Then the monoid presented by
  $\pres{A}{\drel{R}}$ is $\cfwp$, and a context-free grammar generating its word problem can be effectively constructed
  from context-free grammars describing $\drel{R}$.
\end{theorem}

(The preceding result originally stated that a monoid satisfying the hypothesis was word-hyperbolic; however, the proof
proceeds by constructing the word problem for the monoid. The `effective construction' part follows easily by inspecting
the construction in the proof.)

\begin{example}
  \label{eg:nocfcrosssection}
  This example shows that a $\cfwp$ monoid need not have a context-free cross section (that is, a language over some
  generating set containing a unique representative for every element).

  Let $K = \gset{a^\alpha b^\alpha c^\alpha}{\alpha \in \nset\cup\set{0}}$ and let $L = \set{a,b,c}^* - K$. It is
  well-known that $K$ is not a context-free language but that $L$ is a context-free language. Let
  $A = \set{a,b,c,x,y,z}$ and let $\drel{R} = \gset{(xwy,z)}{w \in L}$. Let $M$ be the monoid presented by
  $\pres{A}{\drel{R}}$. By \fullref{Theorem}{thm:contextfreesrs}, $M$ is $\cfwp$. Suppose that $M$
  admits a context-free cross-section. Then $M$ admits a context-free cross-section $J \subseteq A^*$. Let $u$ be the
  unique word in $J$ such that $u =_M z$, and let $J' = (J \setminus \set{u}) \cup \set{z}$; then $J'$ is also a
  context-free cross-section of $M$. Let $H = J' \cap x\set{a,b,c}^*y$. Then $H$ is context-free and comprises precisely
  the words $xwy$ where $w \in K$, for if $w \in L$, then $xwy =_M z$, and the representative of $z$ in $J'$ is the word
  $z$ itself. Hence, since the class of context-free languages is closed under right and left quotients with regular
  sets, $K = x\backslash H/y$ is context-free. This is a contradiction, and so $M$ does not admit a context-free
  cross-section.
\end{example}

\begin{example}
  \label{eg:notdetcfwp}
  This example shows that the class of $U(\DCF)$ semigroups is properly contained in the class of $U(\CF)$ semigroups.

  Let $K$ be the language of palindromes over $\set{a,b}$. It is well-known that $K$ is context-free but not
  deterministic context-free. Let $A = \set{a,b,x,y,z}$ and let $\drel{R} = \gset{(xwy,z)}{w \in L}$. Let $M$ be the
  monoid presented by $\pres{A}{\drel{R}}$.

  By \fullref{Theorem}{thm:contextfreesrs}, $M$ is $\cfwp$. Suppose, with the aim of obtaining a contradiction, that $M$
  is $\dcfwp$. Then $\WP(M,A)$ is deterministic context-free. Let
  $L = (\WP(M,A) \cap A^*\#z)/\set{\#z} \cap \set{a,b,x,y}^*$; then $L$ is the language of words over $\set{a,b,x,y}$
  that are equal to $z$ in $M$. Furthermore, $L$ is deterministic context-free, since the class of deterministic
  context-free languages is closed under intersection with regular languages \cite[Theorem 10.4]{hopcroft_automata} and
  right quotient by regular languages \cite[Theorem 10.2]{hopcroft_automata}.

  Now, $K = x\backslash L/y$. The class of determinstic context-free languages is closed under left quotient by a
  singleton (since a deterministic pushdown automaton can simulate reading a fixed word before it starts reading input),
  and, as noted above, is closed under right quotient by regular languages. Hence $K$ is deterministic
  context-free. This is a contradiction, and so $M$ is not $\dcfwp$.
\end{example}

We conjecture that the bicyclic monoid $B = \pres{b,c}{bc = \emptyword}$, which is $\cfwp$ by
\fullref{Theorem}{thm:contextfreesrs}, is not $\dcfwp$. To motivate this conjecture, consider the intersection of
$\WP(B,\set{b,c})$ with the regular language $c^*\#c^*b^*c^*$. Then a deterministic pushdown automaton recognizing this
intersection, when reading $c^\alpha\#c^\alpha b^\beta c^\gamma$ would have to enter an accept state after reading
$c^\alpha\#c^\alpha$ but could then accept the whole input if and only if $\alpha \geq \beta = \gamma$, and checking two
independent comparisons is heuristically impossible for a pushdown automaton.

\begin{example}
  An example of a monoid that is `close' to being a free group but is not $\cfwp$ is the free inverse monoid of rank $1$
  and hence (by \fullref{Proposition}{prop:changegensubsemigroups}) of any finite rank. This follows from applying the
  pumping lemma to the intersection of the word problem and the regular language $x^*(x^{-1})^*x^*\#x^*$ (where $x$ is
  the free generator); see \cite[Theorem~1]{brough_inverse}.
\end{example}

\section{Relationship to word-hyperbolicity and automaticity}
\label{sec:wordhyp}

Hyperbolic groups have become one of the most fruitful areas of group theory since their introduction by Gromov
\cite{gromov_hyperbolic}. The concept of hyperbolicity can be generalized to semigroups and monoids in more than one
way, but here we consider the linguistic definition that uses Gilman's characterization of hyperbolic groups using
context-free languages \cite{gilman_wordhyperbolic}. A \defterm{word-hyperbolic structure} for a semigroup $S$ is a
pair $(L,M(L))$, where $L$ is a regular language over an alphabet $A$ representing a finite generating set for $S$ such
that $L$ maps onto $S$, and where
\[
M(L) = \{u\#_1v\#_2w^\rev : u,v,w \in L \land uv =_S w\}
\]
(where $\#_1$ and $\#_2$ are new symbols not in $A$) is
context-free.

\begin{theorem}
  Every $\cfwp$ semigroup is word-hyperbolic.
\end{theorem}

The proof is in effect the first paragraph of the proof of \fullref{Theorem}{thm:contextfreesrs} as given in
\cite[Proof of Theorem~2]{cm_wordhypunique}.  We give it here for completeness.

\begin{proof}
  Let $S$ be a $\cfwp$ semigroup, and let $A$ be a finite generating set for $S$.
  Let $\phi : (A \cup \{\#_1,\#_2\})^* \to (A \cup \{\#_1\})^*$ be the homomorphism extending
  \[
    \#_1 \mapsto \emptyword,\quad \#_2 \mapsto \#,\quad a \mapsto a \text{ for all $a \in A$}.
  \]
  Then $M(A^*) = (\WP(S,A))\phi^{-1} \cap (A^*\#_1A^*\#_2A^*)$. Since the class of context-free languages is closed under taking
  inverse homomorphisms, $M(A^*)$ is also context-free. Hence $(A^*,M(A^*))$ is a word-hyperbolic structure for $S$.
\qed
\end{proof}

All hyperbolic groups are automatic \cite[Theorem~3.4.5]{epstein_wordproc}, but word-hyperbolic semigroups may not even
be \emph{asynchronously} automatic \cite[Example~7.7]{hoffmann_relatives}. Even within the smaller class of $\cfwp$
semigroups, one can find semigroups that are not automatic:

\begin{example}
  \label{eg:nonauto}
  Let $A = \{a,b,c,d,z\}$, let $\drel{R} = \{(ab^\alpha c^\alpha d,z) : \alpha \in \nset\}$. Let $M$ be the monoid
  presented by $\pres{A}{\drel{R}}$. Then $M$ is $\cfwp$ by \fullref{Theorem}{thm:contextfreesrs}, but cannot be
  automatic \cite[Corollary~5.5]{campbell_autsg}. (In fact, it can be shown that $M$ is not even asynchronously automatic.)
\end{example}

Given that $\cfwp$ groups are virtually free and thus automatic, and since the monoid in \fullref{Example}{eg:nonauto}
is not cancellative, the following question is natural:

\begin{question}
Is a cancellative $\cfwp$ semigroup necessarily automatic?
\end{question}

\section{Direct products}
\label{sec:dirprod}

A direct product of two finitely generated semigroups is not necessarily finitely generated.  However, a direct product
of two $\cfwp$ semigroups is not necessarily $\cfwp$, even if it is finitely generated: for example, the free monoid of
rank $1$ is $\cfwp$, but the direct product of two copies of this monoid is the free commutative monoid of rank $2$,
which is finitely generated but not $\cfwp$.

For a semigroup $S$, we say that $S$ is \emph{decomposable} if $S^2 = S$.  We will show that for a direct product of two
$\cfwp$ semigroups to be $\cfwp$, it is necessary and sufficient that one of the factors is finite and decomposable
(decomposability being necessary to ensure finite generation).  First we establish sufficiency.

\begin{lemma}
  \label{dirsuf}
  The classes of $\cfwp$ and $\dcfwp$ semigroups are closed under taking direct product with a finite decomposable
  semigroup.
\end{lemma}

\begin{proof}
Let $S$ be a $\cfwp$ semigroup and $T$ a finite decomposable semigroup.
Then $S\times T$ is finitely generated (\cite[Theorem~8.2]{robertson_dirprod}).
Let $C$ be a finite generating set for $S \times T$ and let $A$ and $B$ be
the projections of $C$ onto $S$ and $T$ respectively.  Then $A$ and $B$
are finite generating sets for $S$ and $T$ respectively.
Thus there exists a pushdown automaton $\A$ recognising $\WP(S,A)$,
which can be modified to give a pushdown automaton $\A'$ recognising
$\WP(S\times T, A\times B)$, by processing the symbols from $A$ as usual, while
using the states to record the finite information required to check validity of the
input on the second tape.  Hence $S\times T$ is $\cfwp$.
Moreover, if $S$ is $\dcfwp$, then $\A$ can be taken to be deterministic, in which
case $\A'$ is also deterministc, so $S\times T$ is $\dcfwp$.
\qed
\end{proof}

Necessity arises from the following language-theoretic result, which encapsulates
the idea that context-free languages cannot admit
`cross-dependencies'.
For words $w,w'$, we use the notation $w'\sqsubseteq w$ to mean that $w'$ is a subword of $w$.

\begin{lemma}
  \label{uvuv}
  Let $A$ and $B$ be disjoint alphabets, and let $\rho_A$, $\rho_B$ be equivalence relations on $A^*$ and $B^*$
  respectively with infinitely many equivalence classes.  Then the language
  $L(\rho_A, \rho_B) = \gset{ u_1 v_1 u_2 v_2 } { (u_1,u_2)\in \rho_A, (v_1,v_2)\in \rho_B}$ is not context-free.
\end{lemma}

\begin{proof}
Suppose that $L = L(\rho_A,\rho_B)$ is context-free, and let $k$ be the pumping constant for $L$.
Let $\mathcal{E}_A$ be the set of all equivalence classes of $\rho_A$ that contain a word of length
at most $k$, and define $\mathcal{E}_B$ similarly.

Let $w = u_1v_1u_2v_2\in L$ with $|v_1|, |u_2|>k$.  Then we can write
$w = pqrst$ where $|qrs|\leq k$, $|qs|\geq 1$ and $pq^i rs^i t\in L$ for all $i\in \N_0$.
Due to the form of words in $L$, $q$ and $s$ must each be a subword of some
$u_i$ or $v_i$.  Moreover, the lengths of $u_2$ and $v_1$ preclude the possibility
that $q\sqsubseteq u_1$ and $s\sqsubseteq u_2$ or $p\sqsubseteq v_1$ and $q\sqsubseteq v_2$.
Let $w' = prt = u_1' v_1' u_2' v_2'$.  Then we have $u_i' = u_i$ for some $i\in \set{1,2}$
and $v_j' =  v_j$ for some $j\in \set{1,2}$.  Since $w'\in L$, this implies that
the equivalence classes of the factors are unchanged between $w$ and $w'$.
By induction, we can repeat this process until we obtain a word
$w^\flat = u_1^\flat v_1^\flat u_2^\flat v_2^\flat\in L$ with $|v_1^\flat|\leq k$ or
$|u_2^\flat|\leq k$, where the $u_i^\flat$ are in the same $\rho_A$-equivalence class
as the $u_i$ and the $v_i^\flat$ are in the same $\rho_B$-equivalence class
as the $v_i$.
Hence our original word $w$ had either $u_i\in C$ for some $C\in \mathcal{E}_A$ or
$v_i\in D$ for some $D\in \mathcal{E}_B$.
But $\mathcal{E}_A$ and $\mathcal{E}_B$ are both finite, and so $L$ cannot contain all words
of the form $u_1v_1u_2v_2$ with $(u_1,u_2)\in \rho_A$ and $(v_1,v_2)\in \rho_B$.
Hence $L$ is not context-free.
\qed
\end{proof}

The preceding lemma is immediately applicable only to monoids.

\begin{lemma}
  \label{dirprodmon}
  The direct product of two infinite monoids cannot be $\cfwp$.
\end{lemma}

\begin{proof}
Let $S = \langle A\rangle$ and $T = \langle B\rangle$ be infinite monoids.
Then the relations $\rho_A = \iota(S,A)$ and $\rho_B = \iota(T,B)$ both have
infinitely many equivalence classes.
Moreover, the language $L = \WP(S\times T, A\cup B)\cap A^*B^* \# A^*B^*$
has as a homomorphic image the language $L(\rho_A,\rho_B)$ defined in
\fullref{Lemma}{uvuv}.  Since the class of context-free languages is closed under
homorphisms and intersection with regular sets, this implies that $S\times T$
is not $\cfwp$.

Thus if $S\times T$ is $\cfwp$, then at least one of $S$ or $T$ is finite.
\qed
\end{proof}

In order to extend \fullref{Lemma}{dirprodmon} to all semigroups, we first establish the following fact (which is
clear for monoids, where direct factors are submonoids).

\begin{lemma}
  \label{dirfactor}
  The class of $\cfwp$ semigroups is closed under taking direct factors.
\end{lemma}

\begin{proof}
  Assume that $S \times T$ is $\cfwp$. In particular, $S \times T$ is finitely generated. By
  \cite[Theorem~2.1]{robertson_dirprod}, $S$ and $T$ are finitely generated, and $S^2 = S$ and $T^2 = T$.

  Let $C = \set{c_1,\ldots,c_k}$ be a finite generating set for $T$. Since $T^2 = T$, we can choose a factorization
  $c_i = c_{i\zeta}u_i$ for each $c_i \in C$. Construct a labelled digraph with vertex set $C$ and an edge from $c_i$ to
  $c_{i\zeta}$ labelled by $u_i$ for each $c_i \in C$. Since this digraph is finite, it must contain a circuit. Fix some
  vertex $c$ on that circuit and let $w$ be the concatenation in reverse order of the labels on the edges around the
  circuit. Then $cw = c$.

  Let $A$ be a finite generating set for $S\times T$ and let $B$ be a finite generating set for $S$. Then
  $X = A \cup \parens{B \times \set{c,w}}$ is a finite generating set for $S \times T$. Let $R$ be the regular language
  $(B \times \set{c})(B \times \set{w})^*\#(B \times \set{w})^*(B \times \set{c})$. Let
  $L = \WP(S \times Y,X) \cap R$. Then
  \begin{equation}
    \label{eq:directfactorproof1}
    \begin{aligned}
      &(b_1,c)(b_2,w)\cdots(b_m,w)\#(b'_n,w)\cdots(b'_2,w)(b_1,c) \in L \\
      \iff{}& (b_1b_2\cdots b_m,cw^{m-1}) =_{S\times T} (b'_1b'_2\cdots b'_n,cw^{n-1}) \\
      \iff{}& (b_1b_2\cdots b_m,c) =_{S\times T} (b'_1b'_2\cdots b'_n,c) \\
      \iff{}& b_1b_2\cdots b_m =_{S} b'_1b'_2\cdots b'_n.
    \end{aligned}
  \end{equation}
  Define a homomorphism
  \[
  \pi : \parens[\big]{(B \times \set{c,w}) \cup \set{\#}} \to \parens[\big]{B \cup \set{\#}},\qquad (b,\text{\textvisiblespace}) \mapsto b,\quad \#\mapsto \#.
  \]
  Then \eqref{eq:directfactorproof1} shows that
  $L\pi = \WP(S,B)$. Since the class of context-free languages is closed under homomorphism \cite[Corollary to
  Theorem~6.2]{hopcroft_automata}, $S$ is a $\cfwp$ semigroup.  \qed
\end{proof}

\begin{theorem}
The direct product of two semigroups is $\cfwp$ if and only if it is finite or one of the factors is
$\cfwp$ and the other factor is finite and decomposable.
\end{theorem}
\begin{proof}
Sufficiency was already established in \fullref{Lemma}{dirsuf}.

Conversely suppose that $S\times T$ is $\cfwp$.  Let $C$ be a finite generating set
for $S\times T$ with the projection of $C$ onto the first component being $A$ and
the projection onto the second component $B$.
By \fullref{Lemma}{dirfactor}, $S$ and $T$ are both $\cfwp$.
Let $A_1 = A\times \{1\}$, $B_1 = \{1\}\times B$,
and $C_1 = A_1\cup B_1\cup C$.
We will describe a pushdown automaton $\P$ recognising
$\WP(S^1\times T^1, C_1)$.  This automaton is defined in terms of pushdown automata
$\A$, $\B$ and $\C$, recognising $\WP(S,A)$, $\WP(T,B)$ and
$\WP(S\times T, C)$ respectively.

On input $(x,y)\in C$, the automaton $\P$ behaves as a `delayed' version of $\C$,
storing the input symbol in the state and then (except in the start state, which has no
stored symbol) simulating $\C$ on input of the current stored symbol.
The automaton may guess at any point that the input is complete, and process
the stored symbol from the current state as an $\epsilon$-transition.  In this case
we move to a state with no stored symbol and accepting no further input, which is
a final state if and only if it is a final state in $\C$.  Thus on input in $(C\cup \{\#\})^*$,
$\P$ behaves exactly like $\C$ but `one step behind', and so the sublanguage of
$(C\cup \{\#\})^*$ accepted by $\P$ is $\WP(S\times T, C)$.

In order to work with input from $A_1\cup B_1$ we choose, for all $x,x'\in A$
and $y,y'\in B$, representatives $w_{x,x',y}$ and $w_{x,y,y'}$ in $C$
for the elements $(xx',y)$ and $(x,yy')$ of $S\times T$.

Now, if the automaton $\P$ reads the symbol $(x',1)$ in a state with
stored symbol $(x,y)$, it simulates reading all but the final symbol of
$w_{x,x',y'}$ in $\C$ from the current state, and stores the final symbol
in the last state of this computation.
Symmetrically, the same occurs when we replace $(x',1)$ by $(1,y')$
and $w_{x,x',y}$ by $w_{x,y,y'}$.
Thus on input $u\#v$ from $CC_1^*\# CC_1^*$, the automaton is able to
simulate processing in $\C$ some $u'\# v'$ such that $u =_{S\times T} u'$
and $v =_{S\times T} v'$.

Finally, on input from $A_1$ or $B_1$ in the start state, the automaton
guesses whether the remaining (non-$\#$) input will be in $A_1^*$ resp.\ $B_1^*$.
If it guesses yes, it moves to a copy of the appropriate automaton $\A$
resp.\ $\B$, treating input $(x,1)$ as $x$ and $(1,y)$ as $y$.
Thus the sublanguage of $(A_1\cup \set{\#})^*$ recognised by $\P$ is
$\WP(S\times \{1\}, A_1)$, while the sublanguage of $(B_1\cup \set{\#})^*$
recognised is $\WP(\{1\}\times T, B_1)$.
If, on the other hand, the automaton guesses no, we describe what happens
on input from $A_1$, the other case being symmetric.
Supposing the input is $(x,1)$, the automaton guesses which $y\in B$ will
be read next, and stores this guess in the state, along with the symbol $(x,y)$.
States with a stored guess $y\in B$ operate as usual, except on input
of the form $(x,y)$.  On such input, the automaton deletes the `guess' $y$
and otherwise operates as if the input were $(x,1)$, since it already simulated
reading $y$ earlier.  (If $x=1$, then we simply delete the guess and otherwise do nothing.)
The automaton must similarly make a guess on input from $A_1$ or $B_1$ in
a state with stored symbol $\#$.
Since $(x,1)w(x',y) = (x,y) w (x',1)$ for $w\in A_1^*$, the automaton $\P$ is now
able to simulate reading a corresponding word in $C^*$ for any input not in
$(A_1\cup \{\#\})^*\cup (B_1\cup \{\#\})^*$.  Combined with the fact that $\P$
can also simulate the automata $\A$ and $\B$ on appropriate inputs,
this establishes that $\P$ recognises $\WP(S^1\times T^1, C_1)$.

Thus $S^1\times T^1$ is $\cfwp$, and so by \fullref{Lemma}{dirprodmon},
without loss of generality we can assume $T^1$ is finite.
Moreover, $S^1$ is $\cfwp$, and hence so is $S$, by
Proposition~\ref{prop:changegensubsemigroups}.2.
By \cite[Theorem~8.1]{robertson_dirprod}, if $S$ is infinite
then $T$ must also be decomposable, since $S\times T$ is finitely generated.
\qed
\end{proof}

\section{Free products}
\label{sec:freeprod}

\begin{theorem}
  \label{thm:freeprodsemigroups}
  The class of $\cfwp$ semigroups is closed under taking semigroup free products and under taking free factors.
\end{theorem}

\begin{proof}
  Let $S$ and $T$ be $\cfwp$ semigroups. Let $A_S$ and $A_T$ be finite generating sets for $S$ and $T$, respectively,
  and for $X \in \set{S,T}$, let $\pda{P}_X$ be a pushdown automaton recognizing $\WP(X,A_X)$ accepting by final state,
  Assume that in $\pda{P}_X$, the initial stack content is only a stack bottom symbol
  $\bot_X$, which is never never popped or pushed.

  Construct a new pushdown automaton $\pda{Q}$ recognizing words over $A_S \cup A_T \cup \set{\#}$, functioning as
  follows. First, $\pda{Q}$ will recognize words in $(A_S\cup A_T)^+\#(A_S\cup A_T)^+$; since this is a regular
  language, assume without loss that the input is in this form. When $\pda{Q}$ begins, it reads a symbol from $A_X$ (for
  some $X \in \set{S,T}$). It pushes $\bot_X$ onto its stack and begins to simulate $\pda{P}_X$. Whenever it is
  simulating $\pda{P}_X$ and reads a symbol from $A_Y$, where $Y \neq X$, it pushes the current state of $\pda{P}_X$
  onto the stack, then pushes $\bot_Y$ onto the stack and begins to simulate $\pda{P}_Y$. These alternating simulations
  of $\pda{P}_S$ and $\pda{P}_T$ continue until the $\#$ is encountered.

  On reading the symbol $\#$, the automaton $\pda{Q}$ continues to simulate whichever $\pda{P}_X$ it was currently
  simulating. After this point, whenever it is simulating $\pda{P}_X$ (for some $X \in \set{S,T}$) and reads a symbol
  from $A_Y$, where $Y \neq X$, how it proceeds depends on whether the currently-simulated $\pda{P}_X$ is in an accept
  state:
  \begin{itemize}
  \item If it is in accept state, $\pda{Q}$ pops symbols from its stack until it encounters $\bot_X$, which it pops,
    then pops the state of $\pda{P}_Y$, restores the simulation of $\pda{P}_Y$ from this state (and with the stack
    contents down to the symbol $\bot_Y$), and simulates $\pda{P}_Y$ on reading $\#$ and then on reading the symbol just
    read by $\pda{Q}$. (If after popping $\bot_X$ the stack of $\pda{Q}$ is empty, it fails.)
  \item If it is not in an accept state, $\pda{Q}$ fails.
  \end{itemize}
  These alternating simulations of $\pda{P}_S$ and $\pda{P}_T$ continue until the end of the input unless $\pda{Q}$
  fails before then. At this point $\pda{Q}$ accepts if the currently-simulated $\pda{P}_X$ is in an accept state, and
  if the stack only contains symbols from the stack alphabet $B_X$ plus a single symbol $\bot_X$.

  It follows from the above description that $\pda{Q}$ recognizes strings of the form
  \begin{equation}
    \label{eq:freeprodstring}
    u_1u_2\cdots u_k\#v_k^\rev\cdots v_2^\rev v_1^\rev,
  \end{equation}
  where $u_i\#v_i^\rev \in L(\pda{P}_{X(i)})$ and either $X(2j) = S$ and $X(2j+1) = T$, or else $X(2j) = T$ and
  $X(2j+1) = S$. Thus $\pda{Q}$ recognizes strings \eqref{eq:freeprodstring} such that
  \[
    u_1u_2\cdots u_k =_{S\ast T} v_1v_2\cdots v_k,
  \]
  and the $u_i$ and $v_i$ are either both in $A_X^+$ or both in $A_Y^+$ for alternating $i$. Thus $\pda{Q}$ recognizes
  $\WP(S \ast T,A_X \cup A_Y)$.

  The free factors of a finitely generated free product are themselves finitely generated, so closure
  under free factors follows from Proposition~\ref{prop:changegensubsemigroups}.2.
\qed
\end{proof}

Notice that the strategy of the proof of \fullref{Theorem}{thm:freeprodsemigroups} cannot be applied to show that
the class of $\dcfwp$ semigroups is closed under taking free products. The problem is in the very
last step: after the automaton has read its last input symbols from some $A_X$, it cannot deterministically check that
the stack only contains symbols from the stack alphabet $B_X$ plus a single symbol $\bot_X$. Therefore the following question remains open:

\begin{question}
  Is the class of $\dcfwp$ semigroups closed under forming free products?
\end{question}

\begin{theorem}
  The class of $\cfwp$ monoids is closed under taking monoid free products and free factors.
\end{theorem}

\begin{proof}[Sketch proof]
  It is easy to see that the construction of the $\pda{Q}$ from the proof of
  \fullref{Theorem}{thm:freeprodsemigroups} can be adapted to the case of monoid free products. Using the notation
  from that proof, one observes that for $X \in \set{S,T}$ the language of words over $A_X$ representing the identity of
  $X$ is a context-free language $K_X$. Then one first modifies $\pda{Q}$ to accept $\#$ (that is, the empty word,
  followed by $\#$, followed by the empty word), then modifies $\pda{Q}$ so that it can
  non-deterministically read
  a string from either $K_X$ at any point (including while reading another string from $K_Y$ for $Y \in \set{S,T}$, so
  that such strings can be `nested').
\qed
\end{proof}

\section{Strong semilattices}
\label{sec:strongsemilattices}

We recall the definition of a strong semilattice of semigroups here, and refer the reader to
\cite[Sect.~4.1]{howie_fundamentals} for further background reading:

Let $Y$ be a semilattice. Recall that the meet of $\alpha,\beta \in Y$ is denoted $\alpha \land \beta$. For each $\alpha \in Y$, let
$S_\alpha$ be a semigroup. For $\alpha \geq \beta$, let $\phi_{\alpha,\beta} : S_\alpha \to S_\beta$ be a homomorphism
such that
\begin{enumerate}
\item For each $\alpha \in Y$, the homomorphism $\phi_{\alpha,\alpha}$ is the identity mapping.
\item For all $\alpha,\beta,\gamma \in Y$ with $\alpha \geq \beta \geq \gamma$,
$\phi_{\alpha,\beta}\phi_{\beta,\gamma} = \phi_{\alpha,\gamma}$.
\end{enumerate}
The \defterm{strong semilattice of semigroups} $S =
\mathcal{S}[Y;S_\alpha;\phi_{\alpha,\beta}]$ consists of the disjoint
union $\bigcup_{\alpha \in Y} S_\alpha$ with the following
multiplication: if $x \in S_\alpha$ and $y \in S_\beta$, then
\[
xy = (x\phi_{\alpha,\alpha \land \beta})(y\phi_{\beta,\alpha\land\beta}),
\]
where $\alpha \land \beta$ denotes the greatest lower bound of
$\alpha$ and $\beta$.

\begin{theorem}
  \label{thm:strongsemilattice}
Let $\mathfrak{C}$ be a class of languages closed under finite union,
inverse gsm-mappings and intersection with regular languages
(in particular, the class $\CF$).
A strong semilattice of semigroups is $U(\mathfrak{C})$
if and only if it is finitely generated and all the semigroups in its lattice
are $U(\mathfrak{C})$.
\end{theorem}
\begin{proof}
Let $S = \mathcal{S}[Y;S_\alpha;\phi_{\alpha,\beta}]$ be a strong semilattice of semigroups.
If $S$ is $U(\mathfrak{C})$, then it must be finitely generated (that is, $Y$ must be finite and
each $S_\alpha$ finitely generated).  Moreover, the $S_\alpha$ must all be
$U(\mathfrak{C})$, since they are finitely generated subsemigroups of $S$.

Conversely, suppose that $Y$ is finite and each $S_\alpha$ is $U(\mathfrak{C})$.
For each $\alpha\in Y$, let $A_\alpha$ be a finite generating set for $S_\alpha$
and $A'_\alpha = \bigcup_{\beta\geq \alpha} A_\alpha$.
Let $A = \bigcup_{\alpha\in Y} A_\alpha$.
Define homomorphisms $\phi_\alpha: (A'_\alpha)^*\ra A_\alpha^*$
by $x\mapsto x\phi_{\beta,\alpha}$ for $x\in A_\beta$.

We can view $\WP(S,A)$ as the union of its restrictions to each $S_\alpha$:
that is, as the union of the languages
$L_\alpha = \gset{u\#v^\rev\in \WP(S,A)}{u,v\in S_\alpha}$.
In turn, each $L_\alpha$ can be expressed as $L'_\alpha\cap R_\alpha$, where
$L'_\alpha = \gset{u\#v^\rev}{u,v\in A^*, u\phi_\alpha =_{S_\alpha} v\phi_\alpha}$
and $R_\alpha = \gset{u\#v^\rev}{u,v\in A^*, u, v\in S_\alpha}$.
Note that $u\#v^\rev\in L'_\alpha$ implies $u,v\in S_\beta$ for some
$\beta\geq \alpha$, since otherwise $\phi_\alpha$ is not defined.
We have $L'_\alpha, R_\alpha\subseteq (A'_\alpha)^*$ for all $\alpha\in Y$.


Defining $R'_\alpha = \gset{w\in (A'_\alpha)^*}{w\in S_\alpha}$,
we have $R_\alpha = R'_\alpha\#R'_\alpha$ (since membership of $w$
in $S_\alpha$ depends only on the content of $w$).
The language $R'_\alpha$ is recognised by a finite automaton consisting
of the semilattice $Y$ with an adjoined top element $\top$ as the start state,
and final state $\alpha$.  The transition function is given by the meet operation:
$(\top,x)\mapsto \gamma$ and $(\beta,x)\mapsto \beta\land\gamma$
for $x\in A_\gamma$.  A word $w$ is accepted by this automaton if and only
if the meet of all $\gamma$ such that $w$ contains a symbol
in $A_\gamma$ is $\alpha$.  Thus $R'_\alpha$ is regular,
and hence so is $R_\alpha$, as a concatenation of regular languages.

Now choose a homomorphism
$\psi_\alpha: (A'_\alpha)^*\ra A_\alpha^*$ defined by $x\mapsto w_x$
such that $w_x =_S x\phi_\alpha$.
Let $W = \gset{w_x}{w\in A_\alpha^*}$ and $M = \WP(S_\alpha,A_\alpha)\cap W^*$.
Then $M\in \mathfrak{C}$, and $L'_\alpha$ is the inverse image of $M$ under
the gsm-mapping from $(A'_\alpha)^*\# (A'_\alpha)^*$ to $(A_\alpha)^*\#(A_\alpha)^*$
that preserves $\#$ and maps all symbols in $x$
before the $\#$ to $x\psi_\alpha$ and all symbols $x$ after the $\#$
to $(x\psi_\alpha)^\rev$.  Since $\mathfrak{C}$ is closed
under inverse gsm-mappings, $L'_\alpha$ is thus in $\mathfrak{C}$.
In turn, $L_\alpha$ is in $\mathfrak{C}$,
hence so is $\WP(S,A)$, as the union of the finitely many $L_\alpha$.
\qed
\end{proof}

The class $\DCF$ is not closed under finite union \cite[Theorem~10.5(b)]{hopcroft_automata}.
We conjecture that a finitely generated strong semilattice of $\dcfwp$
semigroups need not be $\dcfwp$. Let $Y = \set{\alpha,\beta}$ be a two-element semilattice with $\alpha > \beta$. Let $S_\alpha$ be the free
group generated by $\set{x,y}$ and let $S_\beta$ be $\zset$ (under $+$). Define $\phi_{\alpha,\beta}$ to be the
homomorphism extending $x \mapsto 1$, $y \mapsto -1$. Both $S_\alpha$ and $S_\beta$ are virtually free groups and so
$\dcfwp$, but the word problem of $\mathcal{S}\bracks[\big]{Y;\set{S_\alpha,S_\beta};\phi_{\alpha,\beta}}$ does not
appear to be deterministic context-free, for checking equality in $S_\alpha$ seems to require computing reduced words on
the stack, while checking equality in $\zset$ seems to require using the stack as a counter, and there is no way to know
in advance which is required.

\section{Rees matrix semigroups}
\label{sec:reesmatrix}

Let us recall the definition of a Rees matrix semigroup. Let $S$ be a semigroup, let $I$ and $\Lambda$ be abstract index
sets, and let $P\in \mathrm{Mat}_{\Lambda\times I}(S)$ (that is, $P$ is a $\Lambda \times I$ matrix with entries from
$S$). Denote the $(\lambda,i)$-th entry of $P$ by $p_{\lambda i}$. The \defterm{Rees matrix semigroup} over $S$ with
sandwich matrix $P$, denoted $\mathcal{M}[S;I,\Lambda;P]$, is the set $I \times S \times \Lambda$ with multiplication
defined by
\[
(i,x,\lambda)(j,y,\mu) = (i,xp_{\lambda j}y,\mu).
\]
This construction is important because it arises in the classification of completely simple semigroups as
Rees matrix semigroups over groups; see \cite[Sect.~3.2--3.3]{howie_fundamentals}.

Hoffmann et al.\ showed that a completely simple semigroup is $\cfwp$ if and only if it is isomorphic to a Rees matrix
semigroup over a finitely generated virtually free group \cite[Theorem~2]{hoffmann_contextfree}; their proof depends on
virtually free groups having \emph{deterministic} context-free word problem. The following theorem generalizes Hoffmann
et al.'s characterization to Rees matrix semigroups over arbitrary semigroups.
See \cite{ginsburg_deterministic,hopcroft_automata} for background on inverse gsm-mappings.

\begin{theorem}
Let $\mathfrak{C}$ be a class of languages closed under inverse gsm-mappings and
intersection with regular languages (in particular, $\CF$ or $\DCF$).
Then a finitely generated Rees matrix semigroup
over a semigroup $S$ is $U(\mathfrak{C})$ if and only if $S$ is $U(\mathfrak{C})$.
\end{theorem}

\begin{proof}
Let $M = \mathcal{M}[S;I,\Lambda;P]$ be a Rees matrix semigroup and let
$\mathfrak{C}$ be as in the statement of the theorem.
If $M$ is $U(\mathfrak{C})$, then it must be finitely generated, hence $S$ is also finitely generated
and thus $U(\mathfrak{C})$.

Conversely, suppose that $S$ is $U(\mathfrak{C})$ and $M$ is finitely generated by
$B\subseteq I\times S\times \Lambda$, and let $A$ be the projection of $B$
onto $S$.  For each $i\in I$ and $\lambda\in \Lambda$, choose a word
$w_{\lambda i}\in A^*$ representing $p_{\lambda i}$.
Let $W$ be the (finite) set of all the $w_{\lambda i}$.
Let $L = \WP(S,A)\cap (AW)^*A\# A(WA)^*$, which is in $\mathfrak{C}$,
as the intersection of a language in $\mathfrak{C}$ with a regular language.
We will define a gsm-mapping $\Phi$ such that $\WP(M,B)$ is the
inverse image of $L$ under $\Phi$.

First, define a gsm-mapping $\phi: B^*\ra A^*$ by
\[(i_1,x_1,\lambda_1)\ldots(i_m,x_m,\lambda_m)\mapsto
x_1 w_{\lambda_1 i_2} x_2\ldots w_{\lambda_{m-1} i_m} x_m.\]
Then for $w = (i_1,x_1,\lambda_1)\ldots(i_m,x_m,\lambda_m)$
we have $w =_M (i(w), w\phi, \lambda(w))$, where $i(w):=i_1$
and $\lambda(w):= \lambda_m$.

Now extend $\phi$ to a gsm-mapping $\Phi: (B\cup \{\#\})^*\ra (A\cup \{\#\})^*$
as follows:
For $u,v\in B^*$ and $w\in (B\cup \{\#\})^*$,
let $(u\# v^\rev)\Phi = u\phi \# (v\phi)^\rev c$, where
$c = \emptyword$ if $i(u) = i(v)$ and $\lambda(u) = \lambda(v)$,
and $c = \#$ otherwise.  (Since $I$ and $\Lambda$ are finite, the computation of
$c$ can be done by storing $i(u)$ and $\lambda(u)$ in the state and then checking
against $\lambda(v)$ and $i(v)$.)
Let $(u\# v^\rev \# w) \Phi = u\phi \# (v\phi)^\rev \#$
(achieved by storing in the state whether $\#$ has already been seen).

The preimage of $L$ in $(B\cup \{\#\})^*$ under $\Phi$ consists of all
words of the form $u\# v^\rev$ with $u,v\in B^*$ such that
$i(u) = i(v)$, $\lambda(u) = \lambda(v)$ and $u\phi \# (v\phi)^\rev\in \WP(S,A)$.
But this is exactly all $u\#v^\rev$ such that $u=_M v$, so
$L\Phi^{-1} = \WP(M, B)$.  Hence $M$ is $U(\mathfrak{C})$, since its word problem is
obtained from a language in $\mathfrak{C}$ by an inverse gsm-mapping.
\qed
\end{proof}


The fact that every completely regular semigroup is isomorphic to a semilattice (not necessarily strong) of completely
simple semigroups \cite[Theorem~
4.1.3]{howie_fundamentals} raises the following question:

\begin{question}
  Which completely regular semigroups are $\cfwp$?
\end{question}

\section{Bruck--Reilly extensions}
\label{sec:bruckreilly}

Let $M$ be a monoid with presentation $\pres{A}{\drel{R}}$
and $\phi : M \to M$ an endomorphism.
The \emph{Bruck--Reilly extension} $\BR(M,\phi)$ of $M$ by $\phi$ is the monoid
with presentation $\pres{ A, b, c }{ \drel{R}, bc=1, ba = (a\phi)b,
ac = c(a\phi)\;\;(a\in A)}$.
This is an analogue for monoids of the notion of HNN-extensions for groups.

If $\phi$ is the identity endomorphism, then $\BR(M,\phi)$ is isomorphic to the
direct product of $M$ with the bicyclic monoid generated by $\set{b,c}$.
Thus by \fullref{Lemma}{dirprodmon} the class of $\cfwp$ semigroups is not
closed under Bruck--Reilly extensions.
In Lemmata~\ref{lem:bruckreillydown} to~\ref{lem:bruckreillyup}
below, we establish a necessary and
sufficient condition for $\BR(M,\phi)$ to be $\cfwp$.

\begin{lemma}
  \label{lem:bruckreillydown}
  Let $M$ be a monoid and $\phi : M \to M$ an endomorphism. If $\BR(M,\phi)$ is $\cfwp$, then $M$ is $\cfwp$.
\end{lemma}

\begin{proof}
  Let $Z$ be a generating set for $\BR(M,\phi)$, and let $Z' = Z \cap M$. Then $M$ is generated by
  \[
    Y = \gset{z\phi^n}{z \in Z', n \in \nset\cup \set{0}}.
  \]
  If every $z \in Z'$ has a finite orbit under $\phi$, then $Y$ is finite and so $M$ is a finitely-generated submonoid
  of $\BR(M,\phi)$ and so $\cfwp$. So suppose, with the aim of obtaining a contradiction, that $z \in Z'$ has infinite
  orbit under $\phi$: that is, that all the elements $z\phi^n$ (where $n \in \nset$) are distinct.

  \begin{lemma}
    Let $\beta,\beta',\delta,\delta' \in \nset \cup \set{0}$. Then there exist
    $\alpha,\alpha',\gamma,\gamma' \in \nset \cup \set{0}$ such that
    $c^\gamma b^\beta zc^\delta b^\alpha =_{\BR(M,\phi)} c^{\gamma'}b^{\beta'}zc^{\delta'}b^{\alpha'}$ if and only if
    $\max\set{\beta,\delta} = \max\set{\beta',\delta'}$.
  \end{lemma}

  \begin{proof}
    The normal form of $c^\gamma b^\beta zc^\delta b^\alpha$ is
    $c^{\gamma-\beta + \max\set{\beta,\delta}}z\phi^{\max\set{\beta,\delta}}b^{\alpha-\delta+\max\set{\beta,\delta}}$;
    the same holds replacing $\gamma,\beta,\alpha,\delta$ by their dashed versions.

    First, suppose that there exist $\alpha,\alpha',\gamma,\gamma'$ such that
    $c^\gamma b^\beta zc^\delta b^\alpha =_{\BR(M,\phi)} c^{\gamma'}b^{\beta'}zc^{\delta'}b^{\alpha'}$. By the previous
    paragraph, since all $z\phi^n$ are distinct, it follows that $\max\set{\beta,\delta} = \max\set{\beta',\delta'}$.

    Now suppose that $\max\set{\beta,\delta} = \max\set{\beta',\delta'}$. Choose $\alpha$ and $\gamma$ such that
    $\gamma = \beta-\max\set{\beta,\delta}$ and $\alpha = \delta - \max\set{\beta,\delta}$, and similarly for dashed
    versions. Then, by the first paragraph,
    $c^\gamma b^\beta zc^\delta b^\alpha =_{\BR(M,\phi)} c^{\gamma'}b^{\beta'}zc^{\delta'}b^{\alpha'}$.
  \qed
\end{proof}

  Let $L$ be the word problem of $\BR(M,\phi)$ with respect to $Z \cup \set{b,c}$; the language $L$ is context-free. Let
  \[
    K = L \cap c^*b^*zc^*b^*\#b^*c^*zb^*c^*.
  \]
  Then $K$ is context-free. Let $J$ be the language obtained from $K$ by applying the gsm-mapping that deletes any
  initial sequence of symbols $c$, any terminal sequence of symbols $c$, any sequence of symbols $b$ immediately before
  $\#$, and any sequence of symbols $b$ immediately after $\#$. Thus $J$ consists of words of the form
  $b^\beta zc^\delta\#c^{\delta'}zb^{\beta'}$ for which there exist
  $\alpha,\alpha',\gamma,\gamma' \in \nset \cup \set{0}$ such that
  $c^\gamma b^\beta zc^\delta b^\alpha\#b^{\alpha'}c^{\delta'}zb^{\beta'}c^{\gamma'} \in K$. Since $K$ is a subset of
  the word problem for $\BR(M,\phi)$, it follows from the lemma that $J$ is the language
  \[
    \gset[\big]{b^\beta zc^\delta\#c^{\delta'}zb^{\beta'}}{\max\set{\beta,\delta} = \max\set{\beta',\delta'}}.
  \]
  A straightforward argument using the pumping lemma shows that $J$ cannot be context-free. However, the class of
  context-free languages is closed under gsm-mappings, and $J$ was obtained from the context-free language $K$ via a
  gsm-mapping. This is a contradiction, which completes the proof.  \qed
\end{proof}

\begin{lemma}
\label{lem:bruckreillyimphi}
  If $\BR(M,\phi)$ is $\cfwp$, then $\im\phi^n$ is finite for some $n$.
\end{lemma}

\begin{proof}
  Suppose $S = \BR(M,\phi)$ is $\cfwp$. As a corollary of the preceding result, $M$ admits a finite generating set $X$,
  and so $S$ is generated by $Y = X \cup \set{b,c}$. Let $L$ be the word problem of $S$ with respect to $Y$. Let
  $K = X^*b^*\#X^*b^* \cap L$; then $K$ is also context-free. Let $n$ be greater than the pumping constant for $K$.

  Let $u \in X^* \cap \im\phi^n$. Then there exists $v \in X^*$ such that $v\phi^n = u$.  Therefore
  $b^nv =_S (v\phi^n)b^n =_S ub^n$. Therefore $ub^n\#v^\rev b^n \in K$.

  Suppose that $|v| > n$.  Applying the pumping lemma shows that $ub^n\#v^\rev b^n$ factors as $pqrst$ with $|r| < n$
  such that $pq^irs^it$ also lies in $K$ for any $i$. By the definition of $K$, the factors $q$ and $s$ must lie wholly
  within the words $u$, $b^n$, $v^\rev$, or $b^n$. By the bound on its length, the factor $r$ cannot involve the whole
  of the left-hand $b^n$ or the whole of $v^\rev$. Further, it is impossible that $q$ and $s$ are both non-empty and lie
  in $b^n$ and $v^\rev$, respectively, for taking $i=2$ gives a word with different numbers of symbols $b$ before and
  after the $\#$. Thus either $q$ lies in $u$ and $s$ is empty, or else $q$ is empty and $s$ lies in $v^\rev$.

  In the former case, taking $i = 0$ yields a word in $K$ of the form $\hat{u}b^n\#v^\rev b^n$, where $|\hat{u}| <
  |u|$. Thus $\hat{u}b^n =_S b^nv =_S (v\phi^n)b^n$, and so $\hat{u} = v\phi^n = u$. Thus we can replace $u$ by a
  strictly shorter word $\hat{u}$ that is equal to it in $S$ (and so in $M$) and repeat the above reasoning.

  In the latter case, taking $i = 0$ yields a word in $K$ of the form $ub^n\#\hat{v}^\rev b^n$, where $|\hat{v}| <
  |v|$. Thus $ub^n =_S b^n\hat{v} = (\hat{v}\phi^n)b^n$. Thus we can replace $v$ by a strictly shorter word $\hat{v}$
  that has same image under $\phi^n$.

  These replacements do not alter the element $u = v\phi^n$ of $\im\phi^n$. However, since these replacements are always
  by strictly shorter words, this replacement process, which began with the assumption that $|v| > n$, must reach a state
  where $|v| \leq n$. That is, $\im\phi^n \subseteq X^{\leq n}\phi^n$ and is thus finite.
\end{proof}

\begin{lemma}
    \label{lem:bruckreillyup}
Let $M$ be a $\cfwp$ monoid and $\phi: M\ra M$ an endomorphism such that
$\im \phi^n$ is finite for some $n$.  Then the Bruck--Reilly extension $\BR(M,\phi)$
is $\cfwp$.
\end{lemma}
\begin{proof}
Let $S = \BR(M,\phi)$.
Let $T = \im \phi^n$, and let $i$ be the index and $p$ the period of the action of $\phi$
on the finite monoid $T$.  Let $m = n + i + p - 1$, so we have $\phi^{m+1} = \phi^{n+i}$.
For $\beta\in \nset_0$, let $[\beta]$ denote the unique element of $\set{0,\ldots,m}$
such that $\phi^\beta = \phi^{[\beta]}$.

Let $X$ be a finite generating set for $M$ with $T\subseteq X$, and let
$\B$ be a pushdown automaton recognising $\WP(M,X)$, with states $Q$,
initial state $q_0$ and stack alphabet $\Gamma$, with bottom marker $\bot\in \Gamma$,
which is never pushed or popped during computations in $\B$.
For $z\in X^*$, we will denote the stack contents (including $\bot$)
of $\B$ after reading
$z$ from the start state $q_0$ by $s_z$, and the state reached at that point
by $q_z$.  Of course if $\B$ is non-deterministic then $s_z$ and $q_z$ are not
necessarily unique.  We use the notation to represent any possible choice.

Let $Y = X\cup \set{b,c}$.  We will define a pushdown automaton $\A$ recognising
$\WP(S,Y)$.  The idea behind the automaton is that while reading a word $u\in Y^*$,
it can use its stack both to record the $c^\alpha$--$b^\beta$ part of
the the Bruck--Reilly normal form, while also simulating $\B$: the
stack contents will be $c^\alpha\text{[simulated stack of
$\B$]}b^\beta$.
The automaton $\A$ tracks of the number of times
$\phi$ must be applied to the next
subword from $X^*$ as a result of uncancelled symbols $b$ already read, and
guesses (and records and later checks) how many further
applications of $\phi$ are required due to future uncancelled symbols $c$.
This guess need only be one of a finite number of possibilities, since there are
only $m$ distinct actions of powers of $\phi$ on $M$.
The key to this working is that whenever the automaton needs to access stack
symbols from $\Gamma$, there will be a bounded amount of information above
these symbols and so they can be accessed without destroying the information above.
If $u =_S c^\alpha \hat{u} b^\beta$, the information the automaton ends up storing
about $\hat{u}$ is comprised of two components: a word $s_z$ in $\Gamma^*$ on the stack,
and a symbol $t$ from $T$ in the state, such that $\hat{u} = zt$.
Upon then reading $\#v^\rev$ for $v\in Y^*$, similar reasoning allows us to check
$v^\rev$ against the stored `normal form' to determine whether $u=_S v$.

The states of $\A$ are divided into `pre-$\#$' and `post-$\#$' states,
and we begin by describing the behaviour in the pre-$\#$ states.  These states are of
the form $(q,t,k,-)$ or $(q,t,k,l)$, where $q\in Q$, $t\in T^1$, $0\leq k \leq m$ and
$l\in \set{0, \ldots, \max\{0,n-k\}}$.
The start state of $\A$ is $(q_0,1,0,-)$, where $1$ is the adjoined identity in $T^1$
(not equal to the identity element of $T$).

Our goal is that after reading a word $w\in Y^*$, for some word $z\in X^*$
the automaton is in the following configuration:
\begin{itemize}
\item
If $w = c^\alpha$, the stack contents are $C^\alpha$.
Otherwise, the stack contents are of the form
$C^\alpha s_z B_\beta$ for $0\leq \beta \leq n+i$ or
$C^\alpha s_z B_{n+i} B^{\beta-n-i}$ for $\beta>n+i$.
\item
If $w\in \set{b,c}^*$, then the state is $(q_0,1,[\beta],-)$.
Otherwise, the state is $(q_z, t, [\beta], l)$, where $l$ represents a guess regarding the
number of symbols $c$ still to be read.  If $l$ takes the maximum value $n-[\beta]$,
the guess is that there are still \emph{at least} that many symbols $c$ to be read,
whereas for other values of $l$ the guess is exact.
\item
If the last state component is blank, then $z = \emptyword $ and $w =_S c^\alpha b^\beta$.
If $l=0$, then $w =_S c^\alpha zt b^\beta$.  Otherwise $w c^{\beta+l} =_S c^\alpha zt$.
\end{itemize}
The state component $k$ records the number of times $\phi$ needs to be applied to
the next upcoming subword from $X^*$ as a result of symbols $b$ already read.
The state component $l$ represents a guess at the number of times $\phi$ needs to
be applied to the same subword as a result of uncancelled symbols $c$ still to be read.
If $\phi$ is to be applied at least $n$ times, then the element represented falls into $T$
and we begin recording it in the state instead, in the component $t$.
The `guess' at the value of $l$ occurs when the automaton first reads a symbol from $X$,
and the idea is that we will be able to tell when we have read
the right number of symbols $c$ by the final state component being $0$.
The automaton has the property that if our initial guess is the maximum possible,
then $t\neq 1$ once the guess has been made.  This allows us to detect when
a non-maximal guess is incorrect:
the automaton fails if a $c$ is read from a state of the form
$(q,1,k,0)$, as this will imply we have seen more symbols $c$ than we had guessed.

We now describe the behaviour of $\A$ in a state $(q,t,k,-)$ or $(q,t,k,l)$ on input from $Y$.
Observe that the stack contents are always in $C^* \Gamma^* S_B$ or in
$C^* \Gamma^* B_{n+i} B^*$, where $S_B = \set{B_1,\ldots,B_{n+i-1}}$.
Note that some transitions technically require more than one state to execute,
but we do not name the intermediate states involved in these transitions,
the important point being that there are finitely many of them in total.
Choose a map $\psi: X\ra X^*$ such that $x\psi =_M x\phi$ for all $x\in X$.
\begin{itemize}
\item
On input $b$: if the stack is empty or the top-of-stack is $C$,
push $\bot B_1$;
if the top-of-stack is in $X$, push $B_1$;
if the top-of-stack is $B_j$ for $j< n+i$, replace it by $B_{j+1}$;
if the top-of-stack is $B_{n+i}$ or $B$, push $B$.
In all cases, replace $k$ by $[k+1]$.
This ensures that after reading $b^\beta$ from the start state,
the third component of the state is $[\beta]$.
\item
On input $c$:
if $B$ at top of stack, pop $B$ and replace $k$ by $k-1$ unless $k = n+i$
and the next symbol on the stack is also $B$, in which case
replace $k$ by $m$;
if $C$ is at the top of the stack, add another $C$ (this can only occur while $t=1$);
if $x\in X$ is at the top of the stack, replace $l$ by $l-1$ unless $l=0$, in which case
we fail if $t=1$, and replace $t$ by $t\phi$.
\item
On input $x\in X$,
if the last state component is blank, we first replace it
by either $0$ if $k\geq n$, or otherwise some non-deterministically chosen
$l\in \set{0,\ldots,n-k}$, before proceeding.
When the last state component is not blank: if $k+l\geq n$, replace $t$ by
$t (x\phi^{[k+l]})\in T$,
without altering the stack; otherwise, the top-of-stack will be $B_k$,
and we pop this from the stack, then simulate reading $x\psi^{k+l}$ from state
$q$ in the automaton $\B$, replacing $q$ by the state reached and finally adding
$B_k$ to the top of the stack.
\end{itemize}

Next, we describe the `transition' phase of $\A$, when $\#$ is read from
a pre-$\#$ state $(q,t,k,-)$ or $(q,t,k,l)$.
\begin{itemize}
\item
If the last state component is blank, this means that we read no symbols
from $X$.  The state will be $(q_0,1,[\beta],-)$ and the stack contents $c^\alpha b^\beta$.
Move to state $(q_\#,1,1,[\beta],0)_\#$.  (Recall that $q_\#$ denotes a state reached after reading
$\#$ in $\B$ from state $q_0$.)
If moreover the top-of-stack is $C$, then we had $u = c^\alpha$ and so the stack contents are $C^\alpha$.
In this case we move to a special state $f$, without altering the stack.
\item
If $l\neq 0$, the automaton fails, since this means we have not read as many
symbols $c$ as we guessed, so that $\phi$ has not been applied the correct number
of times.
\item
If $l=0$, move to state $(q,t,1,[\beta],0)_\#$.
\end{itemize}

Note that after the transition phase, $\A$ is in a `post-$\#$' state of the form
$(q,t,1,k,l)_\#$, or in the state $f$.
In general the post-$\#$ states other than $f$ have the form $(q,t,t',k,l)_\#$
or $(q,t,t',k,l)$ for $q\in Q$, $t,t'\in T^1$, $0\leq l\leq m$ and
$k\in \set{0,\ldots, \max\{0,n-l\}}$.  The subscript $\#$ indicates that the symbol
$\#$ has not yet been simulated in $\B$.
We first describe the action of $b$ and $c$ in these states.
A new set of stack symbols $S_D = \set{D_1,\ldots,D_{n+i}, D}$ is introduced
at this stage to count uncancelled occurrences of $c$.
Note that in these states the bottom-of-stack marker $\bot$ from $\B$ is always
still on the stack, so it is not possible for the top-of-stack to be $C$.

\begin{itemize}
\item
On input $c$:
If top-of-stack is $B$, $B_j$ or $\gamma\in \Gamma$, push $D_1$ onto stack;
if top-of-stack is $D_j$ for $j<n+i$, replace it by $D_{j+1}$;
if top-of-stack is $D_{n+i}$ or $D$, push $D$ onto stack.
In all cases, replace $l$ by $[l+1]$.

\item
On input $b$:
If top-of-stack is $B$ or $B_1$, pop it;
if top-of-stack is $B_j$ for $j>1$, replace it by $B_{j-1}$;
in these two cases replace $k$ by $k-1$, unless $k=n+i$
and the top two stack symbols are both $B$, in which case
replace $k$ by $m$.
If top-of-stack is $D$ or $D_1$, pop it;
if top-of-stack is $D_j$ for $j>1$, replace it by $D_{j-1}$;
in these two cases replace $l$ by $l-1$, or by
$m$ if $l=n+i$ and the top two symbols are both $D$.
The transitions for $B$ and $B_j$ cancel final symbols $b$ from the normal form of $v^\rev$
against those from $u$.  The transitions for $D$ and $D_j$ apply the relation $bc = 1$ in
reverse and ensure that if $v = c^\gamma z b^\delta$ then the final state component is
$[\gamma]$ after reading $v^\rev$.
If top-of-stack is
$\gamma\in \Gamma$, the automaton fails, since this means that the normal form
of $v$ contains more symbols $b$ than the normal form of $u$.
\end{itemize}

Suppose that after reading $u\in Y^*$, the automaton was in state
$(q,t,k,0)$ with stack contents $C^\alpha s_z B^\beta$.
Then after further reading $\# w^\rev$ with $w\in \set{b,c}^*$ and $w=_S c^\gamma b^\delta$,
if $\delta>\beta$ the automaton has failed, and if $\beta\geq\delta$
the stack contents are $C^\alpha s_z B^{\beta - \delta} D^\gamma$
and the state is $(q,t,1,[\beta - \delta], [\gamma])_\#$,
unless we have meanwhile non-deterministically transitioned to the state
$f$ as described below.

Next, we describe the action of $x\in X$ in a post-$\#$ state $(q,t,t',k,l)_{\#}$
or $(q,t,t',k,l)$.
\begin{itemize}
\item
If $k+l\geq n$, replace $t'$ by $(x\phi^{[k+l]})t'\in T$.
\item
If $k+l<n$, the top two symbols of the stack are $B_{\delta} D_{\gamma}$.
Pop these two symbols, then simulate reading $t (\#) t'(x\psi^{[k+l]})$ from
state $q$ in $\B$ (where $(\#)$ indicates that the $\#$ is present if and only
if it was in the subscript of the state), replacing $q$ by the state reached,
and then return the
two popped symbols to the top of the stack (in the same order).
If the state had a $\#$ subscript, remove it to record that $\#$ has now
been simulated in $\B$.
\end{itemize}

Since $\B$ cannot be assumed in general to accept by empty stack, $\A$ has to
non-deterministically guess when it has finished receiving input from $X$:
\begin{itemize}
\item
In each state $(q,t,t',0,l)_{(\#)}$, there is an $\epsilon$-transition
which simulates reading $t(\#)t'$ in $\B$.  Let $q'\in Q$ be the state reached
following this computation.  Then if $q'$ is a final state of $\B$, the automaton
deletes all remaining symbols from $\Gamma$ (including $\bot$) from the stack and
moves to state $f$.
Otherwise the automaton fails.
\end{itemize}

On reaching the state $f$, the stack contents are always of the form $C^\alpha$.
The state $f$ only accepts input $b$ and $c$, and behaves like the `post-$\#$' state of
a pushdown automaton for the word problem of the bicyclic monoid:
\begin{itemize}
\item
On input $c$: with $C$ at top-of-stack, non-deterministically either pop $C$
or push $D$ (to indicate a guess that there are still symbols $b$ to be read);
with $D$ at top-of-stack, or on empty stack, push $D$.
\item
On input $b$: fail if top-of-stack is $C$; on top-of-stack $D$, pop $D$.
\end{itemize}

Final acceptance in $\A$ is by empty stack in state $f$.
\qed
\end{proof}

Lemmata~\ref{lem:bruckreillydown}, \ref{lem:bruckreillyimphi} and~\ref{lem:bruckreillyup} combine to give a complete
characterisation of the $\cfwp$ monoids occurring as Bruck--Reilly extensions:

\begin{theorem}
  \label{thm:bruckreilly}
  Let $M$ be a monoid and $\phi : M \to M$ an endomorphism. Then $\BR(M,\phi)$ is $\cfwp$ if and only if $M$ is $\cfwp$
  and $\im\phi^n$ is finite for some $n$.
\end{theorem}

The bicyclic monoid arises as the submonoid generated by
$\set{b,c}$ in any Bruck--Reilly extension. As discussed in \fullref{Section}{sec:examples}, we conjecture that this monoid is not $U(\DCF)$, and hence no Bruck--Reilly
extension is $U(\DCF)$.

\section{Further open problems}
\label{sec:openprob}

\begin{question}
Does every cancellative $\cfwp$ semigroup have deterministic context-free word problem?
\end{question}

\begin{question}
Is it possible to characterize the commutative (respectively, cancellative, inverse) $\cfwp$ semigroups?
\end{question}

The previous two questions are motivated by the group case, since the classes of $\cfwp$ and $\dcfwp$ groups coincide
and are precisely the virtually free groups. In particular, the abelian $\cfwp$ groups are thus either finite or of the
form $\zset \times F$, where $F$ is finite and abelian.

\begin{question}
Does there exist an infinite periodic $\cfwp$ semigroup?
\end{question}

\begin{question}
Must the group of units of a $\cfwp$ monoid be finitely generated?
\end{question}

\bibliography{\jobname}

\newcommand{\etalchar}[1]{$^{#1}$}
\begin{thebibliography}{HKOT02}

\bibitem[BC18]{bcm_relationlanguages}
T.~Brough \& A.~J. Cain.
\newblock `{A} language hierarchy of binary relations'.
\newblock  2018.
\newblock arXiv:~\href {http://arxiv.org/abs/1805.03125} {{1805.03125}}.

\bibitem[BO93]{book_srs}
R.~V. Book \& F.~Otto.
\newblock {\em {S}tring {R}ewriting {S}ystems}.
\newblock Texts and Monographs in Computer Science. Springer, 1993.

\bibitem[Bro10]{brough_phd}
T.~R. Brough.
\newblock {\em {G}roups with {P}oly-{C}ontext-{F}ree {W}ord {P}roblem}.
\newblock Ph.d. thesis, University of Warwick, 2010.

\bibitem[Bro14]{brough_groups}
T.~Brough.
\newblock `{G}roups with poly-context-free word problem'.
\newblock {\em Groups Complexity Cryptology}, 6, no.~1 (2014), pp. 9--29.
\newblock {\sc doi:} \href {http://dx.doi.org/10.1515/gcc-2014-0002}
  {{10.1515/gcc-2014-0002}}.

\bibitem[Bro18]{brough_inverse}
T.~Brough.
\newblock `{W}ord problem languages for free inverse monoids'.
\newblock In S.~Konstantinidis \& G.~Pighizzini, eds, {\em {D}escriptional
  {C}omplexity of {F}ormal {S}ystems}, no. 10952 in {\em Lecture Notes in
  Computer Science}, pp. 24--36. Springer, 2018.
\newblock {\sc doi:} \href {http://dx.doi.org/10.1007/978-3-319-94631-3_3}
  {{10.1007/978-3-319-94631-3\_3}}.

\bibitem[CM12]{cm_wordhypunique}
A.~J. Cain \& V.~Maltcev.
\newblock `{C}ontext-free rewriting systems and word-hyperbolic structures with
  uniqueness'.
\newblock {\em International Journal of Algebra and Computation}, 22, no.~7
  (2012).
\newblock {\sc doi:} \href {http://dx.doi.org/10.1142/S0218196712500610}
  {{10.1142/S0218196712500610}}.

\bibitem[CRRT01]{campbell_autsg}
C.~M. Campbell, E.~F. Robertson, N.~Ru{\v{s}}kuc, \& R.~M. Thomas.
\newblock `{A}utomatic semigroups'.
\newblock {\em Theoretical Computer Science}, 250, no.~1-2 (2001), pp.
  365--391.
\newblock {\sc doi:} \href {http://dx.doi.org/10.1016/S0304-3975(99)00151-6}
  {{10.1016/S0304-3975(99)00151-6}}.

\bibitem[DG99]{duncan_hyperbolic}
A.~Duncan \& R.~H. Gilman.
\newblock `{W}ord hyperbolic semigroups'.
\newblock {\em Mathematical Proceedings of the Cambridge Philosophical
  Society}, 136, no.~3 (1999), pp. 513--524.
\newblock {\sc doi:} \href {http://dx.doi.org/10.1017/S0305004103007497}
  {{10.1017/S0305004103007497}}.

\bibitem[Dun85]{dunwoody_accessibility}
M.~J. Dunwoody.
\newblock `{T}he accessibility of finitely presented groups'.
\newblock {\em Inventiones Mathematicae}, 81, no.~3 (1985), pp. 449--457.
\newblock {\sc doi:} \href {http://dx.doi.org/10.1007/BF01388581}
  {{10.1007/BF01388581}}.

\bibitem[ECH{\etalchar{+}}92]{epstein_wordproc}
D.~B.~A. Epstein, J.~W. Cannon, D.~F. Holt, S.~V.~F. Levy, M.~S. Paterson, \&
  W.~P. Thurston.
\newblock {\em {W}ord {P}rocessing in {G}roups}.
\newblock Jones {\&} Bartlett, Boston, MA, 1992.

\bibitem[GG66]{ginsburg_deterministic}
S.~Ginsburg \& S.~Greibach.
\newblock `{D}eterministic context free languages'.
\newblock {\em Information and Control}, 9, no.~6 (1966), pp. 620--648.
\newblock {\sc doi:} \href {http://dx.doi.org/10.1016/S0019-9958(66)80019-0}
  {{10.1016/S0019-9958(66)80019-0}}.

\bibitem[Gil02]{gilman_wordhyperbolic}
R.~H. Gilman.
\newblock `{O}n the definition of word hyperbolic groups'.
\newblock {\em Mathematische Zeitschrift}, 242, no.~3 (2002), pp. 529--541.
\newblock {\sc doi:} \href {http://dx.doi.org/10.1007/s002090100356}
  {{10.1007/s002090100356}}.

\bibitem[Gro87]{gromov_hyperbolic}
M.~Gromov.
\newblock `{H}yperbolic {G}roups'.
\newblock In S.~M. Gersten, ed., {\em {E}ssays in {G}roup {T}heory}, no.~8 in
  {\em Mathematical Sciences Research Institute Publications}, pp. 75--263.
  Springer-Verlag, 1987.

\bibitem[HHOT12]{hoffmann_contextfree}
M.~Hoffmann, D.~F. Holt, M.~D. Owens, \& R.~M. Thomas.
\newblock `{S}emigroups with a context-free word problem'.
\newblock In {\em {D}evelopments in {L}anguage {T}heory}, pp. 97--108.
  Springer, 2012.
\newblock {\sc doi:} \href {http://dx.doi.org/10.1007/978-3-642-31653-1_10}
  {{10.1007/978-3-642-31653-1\_10}}.

\bibitem[HKOT02]{hoffmann_relatives}
M.~Hoffmann, D.~Kuske, F.~Otto, \& R.~M. Thomas.
\newblock `{S}ome relatives of automatic and hyperbolic groups'.
\newblock In G.~M.~S. Gomes, J.-{\'{E}}. Pin, \& P.~V. Silva, eds, {\em
  {S}emigroups, algorithms, automata and languages}, p. 379{\textendash }406,
  River Edge, NJ, 2002. World Scientific.
\newblock {\sc doi:} \href {http://dx.doi.org/10.1142/9789812776884_0016}
  {{10.1142/9789812776884\_0016}}.

\bibitem[HOT08]{holt_onecounter}
D.~F. Holt, M.~D. Owens, \& R.~M. Thomas.
\newblock `{G}roups and semigroups with a one-counter word problem'.
\newblock {\em Journal of the Australian Mathematical Society}, 85, no.~02
  (2008), p. 197.
\newblock {\sc doi:} \href {http://dx.doi.org/10.1017/S1446788708000864}
  {{10.1017/S1446788708000864}}.

\bibitem[How95]{howie_fundamentals}
J.~M. Howie.
\newblock {\em {F}undamentals of {S}emigroup {T}heory}.
\newblock No.~12 in {\em London Mathematical Society Monographs: New Series}.
  Clarendon Press, Oxford University Press, New York, 1995.

\bibitem[HR06]{holt_indexedcoword}
D.~F. Holt \& C.~E. R{\"{o}}ver.
\newblock `{G}roups with indexed co-word problem'.
\newblock {\em International Journal of Algebra and Computation}, 16, no.~5
  (2006), pp. 985--1014.
\newblock {\sc doi:} \href {http://dx.doi.org/10.1142/S0218196706003359}
  {{10.1142/S0218196706003359}}.

\bibitem[HRRT99]{holt_cfcoword}
D.~F. Holt, S.~Rees, C.~E. R{\"{o}}ver, \& R.~M. Thomas.
\newblock `{G}roups with context-free co-word problem'.
\newblock {\em Journal of the London Mathematical Society}, 71, no.~3 (1999),
  pp. 643--657.
\newblock {\sc doi:} \href {http://dx.doi.org/10.1112/S002461070500654X}
  {{10.1112/S002461070500654X}}.

\bibitem[HT93]{herbst_onecounter}
T.~Herbst \& R.~M. Thomas.
\newblock `{G}roup presentations, formal languages and characterizations of
  one-counter groups'.
\newblock {\em Theoretical Computer Science}, 112, no.~2 (1993), pp. 187--213.
\newblock {\sc doi:} \href {http://dx.doi.org/10.1016/0304-3975(93)90018-O}
  {{10.1016/0304-3975(93)90018-O}}.

\bibitem[HU79]{hopcroft_automata}
J.~E. Hopcroft \& J.~D. Ullman.
\newblock {\em {I}ntroduction to {A}utomata {T}heory, {L}anguages, and
  {C}omputation}.
\newblock Addison--Wesley, Reading, MA, 1st edition, 1979.

\bibitem[MS83]{muller_contextfree}
D.~E. Muller \& P.~E. Schupp.
\newblock `{G}roups, the theory of ends, and context-free languages'.
\newblock {\em Journal of Computer and System Sciences}, 26, no.~3 (1983), pp.
  295--310.
\newblock {\sc doi:} \href {http://dx.doi.org/10.1016/0022-0000(83)90003-X}
  {{10.1016/0022-0000(83)90003-X}}.

\bibitem[RRW98]{robertson_dirprod}
E.~F. Robertson, N.~Ru{\v{s}}kuc, \& J.~Wiegold.
\newblock `{G}enerators and relations of direct products of semigroups'.
\newblock {\em Transactions of the American Mathematical Society}, 350, no.~07
  (1998), pp. 2665--2686.
\newblock {\sc doi:} \href {http://dx.doi.org/10.1090/S0002-9947-98-02074-1}
  {{10.1090/S0002-9947-98-02074-1}}.

\end{thebibliography}
\bibliographystyle{alphaabbrv}

\end{document}